\documentclass{amsart}
\usepackage{a4wide,mathrsfs,amsmath,amsthm}

\usepackage{tabularx}
\newcommand{\tr}{^\top}
\newcommand{\bbb}{\boldsymbol}
\newcommand{\rr}{\mathbb R}

\newtheorem*{lemma}{Lemma}

\begin{document}

\title[Sherman-Morrison-Woodbury for the Mulitple Shooting Method is a Bad Idea]{Using the Sherman-Morrison-Woodbury Formula to Solve the System of Linear Equations from the Standard Multiple Shooting Method for a Linear Two Point Boundary-Value Problem is a Bad Idea}

\author{Ivo Hedtke}
\address{Mathematical Institute, University of Jena, D-07737 Jena, Germany.}
\thanks{The author was supported by the Studienstiftung des Deutschen Volkes.}
\email{Ivo.Hedtke@uni-jena.de}

\begin{abstract}
We use the standard multiple shooting method to solve a linear two point boundary-value problem. To ensure that the solution obtained by combining the partial solutions is continuous and satisfies the boundary conditions, we have to solve a system of linear equations. Our idea is to first solve a bidiagonal system related to the original system of linear equations, and then update it with the Sherman-Morrison-Woodbury formula. 
We study the feasibility, the numerical stability and the running time of this method. The results are: The method described above has the same stability problems like the well known Condensing method. The running time analysis shows that the new method is slower than the Condensing method. Therefore we recommend not to use the method described in this article. 
\end{abstract}



\maketitle

\section{Introduction}
\noindent We solve the linear two point boundary-value problem
\begin{align*}
  \mathscr L \bbb{x}(t) &:= \dot{\bbb x}(t) - \bbb{A}(t) \bbb{x}(t)= \makebox[20mm][l]{$\bbb{r}(t),$} t\in [a,b]\\
  \mathscr B \bbb{x}(t) &:= \bbb{B}_a\bbb{x}(a) + \bbb{B}_b\bbb{x}(b) = \bbb{\beta}
\end{align*}
with the standard multiple shooting method, where $\bbb x(t)$, $\bbb r(t)\colon [a,b] \to \rr^n$, $\bbb \beta \in \rr^n$, $\bbb A(t)\colon [a,b] \to \rr^{n \times n}$ and $\bbb B_a$, $\bbb B_b \in \rr^{n \times n}$. We divide the interval $[a,b]$ with the shooting points
\begin{gather*}
  a = \tau_0 < \tau_1 < \ldots < \tau_{m-1} < \tau_m = b
\end{gather*}
into $m$ segments $[\tau_j,\tau_{j+1}]$. We use the principle of superposition on each segment to find the solution
\begin{gather*}
  \bbb x_j(t) = \bbb X(t;\tau_j)\bbb c_j + \bbb v(t;\tau_j),
\end{gather*}
where $\bbb c_j$ is a constant vector. $\bbb X(t;\tau_j)$ is a fundamental system which fulfills the IVP
\begin{gather*}
  \mathscr L \bbb X(t;\tau_j) = \bbb 0, ~~ t\in [\tau_j,\tau_{j+1}], \qquad \bbb X(\tau_j;\tau_j) = \bbb I.
\end{gather*}
$\bbb v(t;\tau_j)$ is an inhomogeneous solution of the ODE and fulfills
\begin{gather*}
  \mathscr L \bbb v(t;\tau_j) = \bbb r(t), ~~ t\in [\tau_j,\tau_{j+1}], \qquad \bbb v(\tau_j;\tau_j) = \bbb 0.
\end{gather*}
The problem now consists in determining the vectors $\bbb c_j$ in such a way, that
\begin{enumerate}
 \item the function $\bbb x(t)$ pieced together by the $\bbb x_j(t)$ is continuous and
 \item satisfies the boundary conditions.
\end{enumerate}
We define $\bbb X_j := \bbb X(\tau_{j+1};\tau_j)$ and $\bbb v_j := \bbb v(\tau_{j+1};\tau_j)$. To satisfy the boundary conditions we focus on $\mathscr B \bbb x(t) = \bbb\beta$:
\begin{gather}\label{Rand}
  \bbb B_a \bbb c_0 + \bbb B_b\bbb X_{m-1}\bbb c_{m-1}^{} = \bbb \beta - \bbb B_b\bbb v_{m-1}.
\end{gather}
To ensure that $\bbb x(t)$ is a continuous function we need
\begin{gather*}
  \bbb x_{k-1}(\tau_k) = \bbb x_k(\tau_k), \qquad k=1,\ldots,m-1,
\end{gather*}
which yields to the conditions
\begin{gather}\label{Stetig}
  \bbb c_k - \bbb X_{k-1}\bbb c_{k-1} = \bbb v_{k-1}, \qquad k=1,\ldots,m-1.
\end{gather}
Now we collect equation \eqref{Rand} and the $m-1$ equations \eqref{Stetig} in
the following system of linear equations:
\begin{gather}\label{LGLS}
\bbb M\bbb c=\bbb q,
\end{gather}
where we define $\bbb Y_j := -\bbb X_j$ and
\begin{gather*}
\bbb M:=\begin{bmatrix}
          \bbb Y_0 & \bbb I\\
          & \bbb Y_1 & \bbb I\\
          &&\ddots & \ddots\\
          &&& \bbb Y_{m-2} & \bbb I\\
          \bbb B_a & & & & \bbb B_b\bbb X_{m-1}
         \end{bmatrix}, ~  
\bbb c:=\begin{pmatrix}
     \bbb c_0\\\bbb c_1\\\vdots\\\bbb c_{m-1}
    \end{pmatrix}, ~
\bbb q:= \begin{pmatrix}
     \bbb v_0\\
     \vdots\\
     \bbb v_{m-2}\\
     \bbb \beta  - \bbb B_b\bbb v_{m-1}
    \end{pmatrix}.
\end{gather*}
Note that $\bbb c$, $\bbb q \in \rr^{mn}$ and $\bbb
M \in \rr^{mn \times mn}$. It is known that $\bbb M$ is regular if
we assume that the BVP has an unique solution. In this case
\begin{gather}\label{ESV}
\bbb N:= \bbb B_a + \bbb B_b\bbb X(b;a)
\end{gather}
is regular, too. (see \cite[Satz 8.1 (Theorem 8.1)]{hermann})

\section{The aim of this work} 
\noindent There exists the well known method \emph{Condensing} to solve the system \eqref{LGLS} (see Section \ref{Condensing}). Because of the special structure of $\bbb M$ it is pretty obvious to try to find the solution in the following way: First solve the bidiagonal system from \eqref{bidiag} and then update the solution with the Sherman-Morrison-Woodbury formula.
In this paper we study the feasibility, the numerical stability and the running time of this method.

\section{The Sherman-Morrison-Woodbury formula} 
\noindent Let $\bbb A$ be a regular $\ell
\times \ell$ matrix and $\bbb U$ and $\bbb V$ be two $\ell \times p$ matrices.
If $\bbb I_p+\bbb V\tr \bbb A^{-1}\bbb U$ is regular, then
\begin{gather}\label{Sherman}
 (\bbb A+\bbb U\bbb V\tr)^{-1} = \bbb A^{-1} - \bbb A^{-1}\bbb U
 (\bbb I_p+\bbb V\tr \bbb A^{-1}\bbb U)^{-1}\bbb V\tr \bbb A^{-1}.
\end{gather}
holds.

\section{Is it possible to use the Sherman-Morrison-Woodbury formula to solve $\bbb M \bbb c = \bbb q$?}

\noindent First, we have to split $\bbb M$ into two matrices $\bbb M = \mathcal M +
\mathcal U$,
where $\mathcal U$ can be written in the form $\mathcal U =\bbb U\bbb V\tr$
with $\bbb U$, $\bbb V$ $\in \rr^{mn\times n}$. For this we define 
\begin{gather*}
   \bbb U = \left[\bbb 0,\ldots,\bbb 0, \bbb B_a\tr\right]\tr \qquad \text{and}
\qquad
   \bbb V\tr = [\bbb I_n,\bbb 0,\ldots,-\bbb L],
\end{gather*}
where $\bbb L:= \bbb X_0^{-1}\cdots \bbb X_{m-2}^{-1}$. Therefore we have
\begin{gather*}
\mathcal U =  \bbb U\bbb V\tr =\begin{bmatrix}
          \bbb 0&\cdots &\bbb 0\\
          &\ddots\\
          \bbb B_a & & -\bbb B_a \bbb L
              \end{bmatrix},
\end{gather*}
and
\begin{gather}\label{bidiag}
 \mathcal M = \bbb M - \mathcal U =
         \begin{bmatrix}
          \bbb Y_0 & \bbb I\\
          & \bbb Y_1 & \bbb I\\
          &&\ddots & \ddots\\
          &&& \bbb Y_{m-2} & \bbb I\\
          & & & & \mathcal B
         \end{bmatrix},
\end{gather}
where $\mathcal B := \bbb B_b\bbb X_{m-1} + \bbb B_a \bbb L$.

Now we have to check that $\mathcal M$ is regular. Because of
\begin{gather*}
\det \mathcal M = \det\mathcal B\prod\nolimits_{j=0}^{m-2}\det \bbb Y_j,
\end{gather*}
it follows that $\det \mathcal M \neq 0$ iff $\det \mathcal B \neq 0$, because
the $\bbb Y_j$ are fundamental systems. But $\mathcal B = \bbb N \bbb L$ and
$\bbb N$ and $\bbb L$ are both regular. This follows from\enlargethispage{\baselineskip}
\begin{gather}\label{Xba}
\bbb X(b;a) = \prod\nolimits_{j=1}^m \bbb X_{m-j}.
\end{gather}
This shows that $\mathcal M$ is regular.

Finally we have to check that $\bbb I_n+\bbb V\tr \mathcal M^{-1}\bbb U$ is
regular. First we need an auxiliary result:

 \begin{lemma}
 Given $m$ regular $n \times n$ matrices $\bbb D_i$. Then, the matrix
  \begin{gather*}
   \Delta := \begin{bmatrix}
         \bbb D_0 & \bbb I_n\\
         & \bbb D_1 & \bbb I_n\\
         & & \ddots & \ddots\\
         & & & \bbb D_{m-2} & \bbb I_n\\
         & & & & \bbb D_{m-1}
        \end{bmatrix}
  \end{gather*}
  is regular and
  \begin{gather*}
   \Delta^{-1} = \begin{bmatrix}
         \bbb D_0^{-1} & -(\bbb D_1\bbb D_0)^{-1} & (\bbb D_2\bbb D_1\bbb
D_0)^{-1} & \ldots &
(-1)^{m-1}(\bbb D_{m-1}\cdots \bbb D_0)^{-1}\\
         & \bbb D_1^{-1} & -(\bbb D_2\bbb D_1)^{-1}\\
         & & \ddots\\
         & & & \bbb D_{m-2}^{-1} & -(\bbb D_{m-1} \bbb D_{m-2})^{-1}\\
         & & & & \bbb D_{m-1}^{-1}
        \end{bmatrix}
  \end{gather*}
  holds.
 \end{lemma}
 
 \begin{proof}
 It holds $\det \Delta = \prod_{j=0}^{m-1} \det \bbb D_j \neq 0$.
$\Delta\Delta^{-1}=\bbb I_{mn}$ and $\Delta^{-1}\Delta = \bbb I_{mn}$ can easily
be
verified.
 \end{proof}
 
Now we go back to the matrix $\bbb I_n+\bbb V\tr \mathcal M^{-1}\bbb U$. With
$\mathcal M^{-1}_j$ we denote the $j$th column of $\mathcal M^{-1}$ and we
write $\mathcal M^{-1}_{ij}$ for the $n\times n$ sub-matrix in the $i$th row and
$j$th column of $\mathcal M^{-1}$. With the lemma above and the new notation we
get
\begin{align*}
\bbb V\tr \mathcal M^{-1}\bbb U &= [\bbb I_n,\bbb 0,\ldots,\bbb 0,-\bbb L
][\mathcal M_1^{-1} \, | \, \ldots \, | \, \mathcal M_m^{-1}]
\left[\bbb 0,\ldots,\bbb 0, \bbb B_a\tr\right]\tr\\
&= [ \mathcal M^{-1}_{11} - \bbb L \mathcal M^{-1}_{m1} \, | \, \ldots \, | \,
 \mathcal M^{-1}_{1m} - \bbb L \mathcal M^{-1}_{mm}]\left[\bbb 0,\ldots,\bbb 0,
\bbb B_a\tr\right]\tr\\
&= \mathcal M^{-1}_{1m}\bbb B_a - \bbb L \mathcal M^{-1}_{mm}\bbb B_a.
\end{align*}
With the special structure of $\mathcal M^{-1}$ we can calculate the two
sub-matrices $\mathcal M^{-1}_{1m}$ and $\mathcal M^{-1}_{mm}$ very easy:
$\mathcal M^{-1}_{mm} = \mathcal B^{-1}$ and
\begin{align*}
\mathcal M^{-1}_{1m} &= (-1)^{m-1} \left(\mathcal B\prod_{j=2}^m \bbb
Y_{m-j}\right)^{-1} = (-1)^{m-2} \bbb Y_0^{-1} \cdots \bbb Y_{m-2}^{-1}
\mathcal B^{-1}\\
&= \bbb X_0^{-1} \cdots \bbb X_{m-2}^{-1} \mathcal B^{-1} = \bbb L \mathcal
B^{-1}.
\end{align*}
Now it follows that
\begin{gather*}
\bbb V\tr \mathcal M^{-1}\bbb U = \mathcal M^{-1}_{1m}\bbb B_a - \bbb L \mathcal
M^{-1}_{mm}\bbb B_a = \bbb L \mathcal B^{-1}\bbb B_a - \bbb L \mathcal
B^{-1}\bbb B_a = \bbb 0.
\end{gather*}
The result above shows that $\bbb I_n+\bbb V\tr \mathcal M^{-1}\bbb U=\bbb I_n$
is regular and we can use the Sherman-Morrison-Woodbury formula to solve
\eqref{LGLS}.

\section{Solving $\mathcal M \bbb c = \bbb q$ with the
Sherman-Morrison-Woodbury formula}
\noindent With \eqref{Sherman} the solution of \eqref{LGLS} can now be expressed as
\begin{align*}
\bbb c &= \bbb M^{-1}\bbb q = (\mathcal M + \mathcal U)^{-1}\bbb q = \mathcal M
^{-1}\bbb q -
\mathcal M^{-1} \bbb U (\bbb I_n + \bbb V\tr \mathcal M^{-1} \bbb U)^{-1}\bbb
V\tr \mathcal M^{-1}\bbb q\\
&=\mathcal M ^{-1}\bbb q - \mathcal M^{-1} \bbb U\bbb V\tr \mathcal M^{-1}\bbb
q = \mathcal M ^{-1}\bbb q - \mathcal M^{-1} \mathcal U \mathcal M^{-1}\bbb
q.
\end{align*}
This gives us an algorithm to solve \eqref{LGLS}:
\begin{enumerate}
 \item Solve $\mathcal M \bbb \xi = \bbb q$.
 \item Solve $\mathcal M \bbb \zeta = \mathcal U \bbb \xi$.
 \item Calculate $\bbb c = \bbb \xi - \bbb \zeta$.
\end{enumerate}
First we study the problem (1.) in detail. We have to solve
\begin{gather*}
\begin{bmatrix}
          \bbb Y_0 & \bbb I\\
          & \bbb Y_1 & \bbb I\\
          &&\ddots & \ddots\\
          &&& \bbb Y_{m-2} & \bbb I\\
          & & & & \mathcal B
         \end{bmatrix}
         \begin{bmatrix}
\bbb \xi_0 \\ \bbb \xi_1 \\ \vdots \\ \bbb \xi_{m-2} \\ \bbb \xi_{m-1}
\end{bmatrix}
=
\begin{bmatrix}
\bbb q_0 \\\bbb q_1 \\ \vdots \\\bbb q_{m-2} \\\bbb q_{m-1}
\end{bmatrix}.
\end{gather*}
Therefore we solve $\mathcal B \bbb \xi_{m-1} = \bbb q_{m-1}$ and use recursion
to find the other $\bbb\xi_j$:
\begin{gather*}
\bbb Y_j \bbb \xi_j = \bbb q_j - \bbb \xi_{j+1}, \qquad j = m-2,\ldots,0.
\end{gather*}
We use the same method for our problem (2.). After we calculated
\begin{gather*}
\mathcal U \bbb \xi = \begin{bmatrix}
          \bbb 0&\cdots &\bbb 0\\
          &\ddots\\
          \bbb B_a & & -\bbb B_a \bbb L
              \end{bmatrix}
\begin{bmatrix}
\bbb \xi_0\\ \vdots\\ \bbb \xi_{m-1}
\end{bmatrix}
=
\begin{bmatrix}
\bbb 0\\ \vdots \\ \bbb 0 \\ \bbb B_a (\bbb \xi_0 - \bbb L \bbb \xi_{m-1})
\end{bmatrix},
 \end{gather*}
the resulting system of linear equations is
\begin{gather*}
\begin{bmatrix}
          \bbb Y_0 & \bbb I\\
          & \bbb Y_1 & \bbb I\\
          &&\ddots & \ddots\\
          &&& \bbb Y_{m-2} & \bbb I\\
          & & & & \mathcal B
         \end{bmatrix}
         \begin{bmatrix}
\bbb \zeta_0 \\ \bbb \zeta_1 \\ \vdots \\ \bbb \zeta_{m-2} \\ \bbb \zeta_{m-1}
\end{bmatrix}
=
\begin{bmatrix}
\bbb 0\\\bbb 0\\ \vdots \\ \bbb 0 \\ \bbb B_a (\bbb \xi_0 - \bbb L \bbb
\xi_{m-1})
\end{bmatrix}.
\end{gather*}
Again we first solve $\mathcal B \bbb \zeta_{m-1} = \bbb B_a (\bbb \xi_0 - \bbb
L \bbb \xi_{m-1})$ and then solve the remaining systems of linear equations
with recursion:
\begin{gather*}
\bbb Y_j \bbb \zeta_j = - \bbb\zeta_{j+1}, \qquad j = m-2,\ldots,0.
\end{gather*}

\section{Condensing}\label{Condensing}
\noindent We want to compare the new method above with the well known standard method from Stoer and Bulirsch. They solve \eqref{LGLS} in the following way (see \cite{Deuflhard} or \cite{StoerBulirsch}):
\begin{enumerate}
 \item Compute $\bbb E:= \bbb B_a + \bbb B_b\bbb X_{m-1}\cdots\bbb X_0$ and $\bbb u := \bbb q_{m-1} - \bbb B_b \bbb X_{m-1}(\bbb q_{m-2} + \bbb X_{m-2}\bbb q_{m-3} + \cdots + \bbb X_{m-2}\cdots \bbb X_1 \bbb q_0)$.
 \item Solve $\bbb E \bbb c_0 = \bbb u$.
 \item Compute the remaining $\bbb c_j$ with recursion: $\bbb c_{j+1} = \bbb q_j + \bbb X_j \bbb c_j$.
\end{enumerate}
In the first step of our new algorithm from the section above we solve $\mathcal B \bbb \xi_{m-1} = \bbb q_{m-1}$. Notice that $\mathcal B = \bbb N \bbb L$. But $\bbb N = \bbb E$ holds. This follows directly from \eqref{ESV} and \eqref{Xba}. That means our new algorithm has the same stability problems like the Condensing method. See \cite{Deuflhard} and \cite{Hermann75} for a detailed discussion.

Therefore we only analyse the number of flops used by the two algorithms to compare them.

\section{Running time analysis}
\noindent We use LU-factorization to solve the systems of linear equations. We assume that this needs $2/3n^3$ flops for a $n\times n$ system.

\begin{table}
\caption{Running time analysis for the Condensing method.}
\small
\begin{tabularx}{\linewidth}{lXl}
\hline
step & description & flops\\\hline\hline
1 & Compute $\bbb E$ and $\bbb u$. Because we compute the products of the $\bbb X_j$  matrices in $\bbb E$ we can use them to compute $\bbb u$, too. Therefore we need no extra product computations of matrices to compute $\bbb u$.\\
& \textbullet{} $m-1$ matrix-matrix multiplications for $\bbb E$ &$(m-1)(2n^3-n^2)$\\
& \textbullet{} one matrix addition for $\bbb E$ & $n^2$\\
& \textbullet{} $m-1$ matrix-vector products for $\bbb u$ & $(m-1)(2n^2-n)$\\
& \textbullet{} $m$ vector additions for $\bbb u$& $mn$\\\hline
2 & Solve $\bbb E \bbb c_0 = \bbb u$.& $2/3n^3$\\\hline
3 & Compute the remaining $\bbb c_j$ with recursion.\\
& \textbullet{} $m-1$ matrix-vector products & $(m-1)(2n^2-n)$\\
& \textbullet{} $m-1$ vector additions & $(m-1)n$
\\\hline\hline
$\sum$ & \multicolumn{2}{l}{$=2mn^3+3mn^2-4/3n^3-2n^2+n$ flops}\\
 \hline
\end{tabularx}
\label{tab:Condensing}
\end{table}

\begin{table}
\caption{Running time analysis of our new method.}
\small
\begin{tabularx}{\linewidth}{lXl}
\hline
step & description & flops\\\hline\hline
1 & Solve $\mathcal M \bbb \xi = \bbb q$.\\
1.1 & Solve $\mathcal B \bbb \xi_{m-1} = \bbb q_{m-1}$.\\
& \textbullet{} Compute $\bbb T:=\bbb L^{-1} = \bbb X_{m-2}\cdots \bbb X_0$. & $(m-2)(2n^3-n^2)$\\
& \textbullet{} Compute $\bbb N := \bbb B_a + \bbb B_b \bbb X_{m-1}\bbb T$. & $4n^3-n^2$\\
& \textbullet{} Solve $\bbb N \bbb s = \bbb q_{m-1}$. & $2/3n^3$\\
& \textbullet{} Compute $\bbb \xi_{m-1} = \mathcal B^{-1}\bbb q_{m-1} = \bbb L^{-1} \bbb N^{-1} \bbb q_{m-1} = \bbb T \bbb s$. & $2n^2-n$\\
1.2 & Use recursion to find the other $\bbb\xi_j$. & $(m-2)(2/3n^3+n)$\\
\hline
2 & Solve $\mathcal M \bbb \zeta = \mathcal U \bbb \xi$.\\
2.1 & Solve $\mathcal B \bbb \zeta_{m-1} = \bbb B_a (\bbb \xi_0 - \bbb
L \bbb \xi_{m-1})$.\\
& \textbullet{} Solve $\bbb T \bbb t = \bbb \xi_{m-1}$. & $2/3n^3$\\
& \textbullet{} Compute $\bbb B_a(\bbb \xi_0 - \bbb t)$. & $2n^2$\\
& \textbullet{} Solve $\bbb N \tilde{\bbb s} = \bbb B_a(\bbb \xi_0 - \bbb t)$.& $2/3n^3$ \\
& \textbullet{} Compute $\bbb \zeta_{m-1} = \bbb T\tilde{\bbb s}$. & $2n^2-n$\\
2.2 & Use recursion to find the other $\bbb\zeta_j$. & $(m-2)(2/3n^3)$\\
\hline
3 & Compute $\bbb c = \bbb \xi - \bbb \zeta$. & $mn$
\\\hline\hline
$\sum$ & \multicolumn{2}{l}{$=10/3mn^3 -mn^2+mn -2/3n^3+7n^2-4n$ flops}\\
 \hline
\end{tabularx}
\label{tab:New}
\end{table}

The running time of the Condensing method is analyzed in Table \ref{tab:Condensing}. For a running time analysis of our new method see Table \ref{tab:New}. The result is: The Condensing method is faster than the new method described above.

\section{Conclusion} We found a new algorithm to solve the system of linear equations from the boundary and continuity conditions with the Sherman-Morrison-Wood\-bury formula. This new method has the same stability problems like the Condensing method. Our new method is also slower than the Condensing method. Therefore it is not recommendable to use the Sherman-Morrison-Woodbury formula in this case.

\end{document}